\documentclass{amsart}
\usepackage{graphicx}
\usepackage{pgfplots}
\usepackage{url}
\usepackage{mathtools}
\usepackage{amssymb}
\usepackage{amsmath}
\newtheorem{theorem}{Theorem}[section]
\newtheorem{prop}[theorem]{Proposition}

\newtheorem{alg}[theorem]{Algorithm}
\newtheorem{eg}[theorem]{Example}

\usepackage[mathscr]{euscript}
 \let\mathscr\relax
\usepackage[scr]{rsfso}

\numberwithin{equation}{section}
\usepackage{enumerate}
\newcommand{\bbm}{\begin{bmatrix}}
	\newcommand{\ebm}{\end{bmatrix}}
\newcommand{\be}{\begin{equation}}
	\newcommand{\ee}{\end{equation}}
\newcommand{\baray}{\begin{array}}
	\newcommand{\earay}{\end{array}}

	\newcommand{\re}{\mathbb{R}}

	\newcommand{\N}{\mathbb{N}}

	\newcommand{\eps}{\epsilon}

	\newcommand{\dt}{\delta}

	\def\af{\alpha}

	\def\gm{\gamma}

	\newcommand{\sig}{\sigma}

	\newcommand{\reff}[1]{(\ref{#1})}

	\newcommand{\mc}[1]{\mathcal{#1}}

	\newcommand{\RN}[1]{%
		
		\textup{\uppercase\expandafter{\romannumeral#1}}%
		
	}

\begin{document}
	
	\title{Polynomial Optimization Over Unions of Sets}

	\author{Jiawang Nie}
	\address{Jiawang Nie, Linghao Zhang, Department of Mathematics, University of California San Diego,
     9500 Gilman Drive, La Jolla, CA, USA, 92093.}
	\email{njw@math.ucsd.edu, liz010@ucsd.edu}

	\author{Linghao Zhang}

	\begin{abstract}
		This paper studies the polynomial optimization problem whose feasible set is a union of several basic closed semialgebraic sets. We propose a unified hierarchy of Moment-SOS relaxations to solve it globally. Under some assumptions, we prove the asymptotic or finite convergence of the unified hierarchy. Special properties for the univariate case are discussed.
		The application for computing $(p,q)$-norms of matrices is also presented.
	\end{abstract}

	\maketitle

	\section{Introduction}
	We consider the optimization problem
	\begin{equation}\label{1.1}
		\left\{ \baray{cl}
			\min & f(x)\\
			s.t. & x \in K \coloneqq \bigcup\limits_{l=1}^m K_l,
		\earay \right.
	\end{equation}
	where each $K_l$ is the basic closed semialgebraic set given as
	\[
	K_l = \left\{ x \in \mathbb{R}^n \ \middle\vert \begin{array}{l}
		c_i^{(l)}(x) = 0 ~ (i \in \mathcal{E}^{(l)}), \\
			c_j^{(l)}(x) \geq 0 ~ (j \in \mathcal{I}^{(l)})
	\end{array} \right\}.
	\]
	Here, all functions $f$, $c_i^{(l)}$, $c_j^{(l)}$ are polynomials in $x \coloneqq (x_1, \ldots, x_n)$; all $\mathcal{E}^{(l)}$ and $\mathcal{I}^{(l)}$ are finite labeling sets. We aim at finding the global minimum value $f_{min}$ of (\ref{1.1}) and a global minimizer $x^*$ if it exists.
	It is worthy to note that solving \reff{1.1} is equivalent to solving $m$ standard polynomial optimization problems by minimizing $f(x)$ over each $K_l$ separately, for $l = 1, \ldots, m$.
	When $K$ is nonempty and compact, $f_{min}$ is achievable at a feasible point,  and (\ref{1.1}) has a minimizer. When $K$ is unbounded, a minimizer may or may not exist. We refer to \cite[Section~5.1]{nie2023moment} for the existence of optimizers when the feasible set is unbounded.

	The optimization (\ref{1.1}) contains a broad class of problems.
	For the case $m=1$,
	if all functions are linear, then (\ref{1.1}) is a linear program (LP); if $f$ is quadratic and all $c_i^{(l)}, c_j^{(l)}$ are linear, then (\ref{1.1}) is a quadratic program (QP); if all $f, c_i^{(l)}, c_j^{(l)}$ are quadratic, then (\ref{1.1}) is a quadratically constrained quadratic program (QCQP).
	Polynomial optimization has wide applications, including
	combinatorial optimization \cite{deKbook, LauICM},
	optimal control \cite{HKL20},
	stochastic and robust optimization \cite{NYZ20, NYZZ23, ZCN23},
	generalized Nash equilibrium problems \cite{NTNEP, NTConvex, NTZ23},
	and tensor optimization \cite{DNY22, njwSTNN17, NTYZDehom, NieZhang18}.

	When the feasible set $K$ is a single basic closed semialgebraic set (i.e., $m=1$) instead of a union of several ones, the problem (\ref{1.1}) becomes a standard polynomial optimization problem. There exists much work for solving standard polynomial optimization problems. A classical approach for solving them globally is the hierarchy of Moment-SOS relaxations \cite{lasserre2001global}. Under the archimedeanness for constraining polynomials, this hierarchy gives a sequence of convergent lower bounds for the minimum value $f_{min}$. The Moment-SOS hierarchy has finite convergence if the linear independence constraint qualification, the strict complementarity and the second order sufficient conditions hold at every global minimizer \cite{NieOPCD}.
	When the equality constraints define a finite set, this hierarchy is also tight \cite{LLR08,Lau07,nie2013polynomial}.
	We refer to the books and surveys \cite{HKL20, DH23, LasBk15, Lau09, nie2023moment} for introductions to polynomial optimization.

	\subsection*{Contributions}
	When $m > 1$, the difficulty for solving the optimization problem (\ref{1.1}) increases.
	A straightforward approach to solve \reff{1.1} is to minimize $f(x)$ over each $K_l$ separately,	for $l = 1, \ldots, m$. By doing this, we reduce the problem (\ref{1.1}) into $m$ standard polynomial optimization problems.
	
	In this paper, we propose a unified Moment-SOS hierarchy for solving (\ref{1.1}). The standard $k$th order moment relaxation for minimizing $f(x)$ over the subset $K_l$ is (for $l = 1, \ldots, m$)
	\begin{equation}\label{1.3}
		\left\{ \baray{cl}
			\min & \langle f \; , \; y^{(l)}\rangle \\
			s.t. & \mathscr{V}_{c^{(l)}_i}^{(2k)}[y^{(l)}] = 0 ~ (i \in \mathcal{E}^{(l)}), \\
			& L_{c^{(l)}_j}^{(k)}[y^{(l)}] \succeq 0 ~ (j \in \mathcal{I}^{(l)}), \\
			& M_k[y^{(l)}] \succeq 0, \\
			& y^{(l)}_0 = 1, y^{(l)} \in \mathbb{R}^{\mathbb{N}^n_{2k}}.
		\earay \right.
	\end{equation}
	We refer to Section 2 for the above notation.
	The unified moment relaxation we propose in this paper is
	\begin{equation}\label{1.5}
		\left\{ \baray{cl}
			\min & \langle f \; , \; y^{(1)}\rangle + \cdots + \langle f \; , \; y^{(m)}\rangle \\
			s.t. & \mathscr{V}_{c^{(l)}_i}^{(2k)}[y^{(l)}] = 0 ~ (i \in \mathcal{E}^{(l)}), \\
			& L_{c^{(l)}_j}^{(k)}[y^{(l)}] \succeq 0 ~ (j \in \mathcal{I}^{(l)}), \\
			& M_k[y^{(l)}] \succeq 0, \sum\limits_{l=1}^{m}y_0^{(l)}=1, \\
			&  y^{(l)} \in \mathbb{R}^{\mathbb{N}^n_{2k}},  l = 1, \ldots, m. \\
		\earay \right.
	\end{equation}
	For $k = 1, 2, \ldots$, this gives a unified hierarchy of relaxations.
	
	A major advantage of \reff{1.5} is that it gives a unified convex relaxation for solving \reff{1.1} instead of solving it over each $K_l$ separately.
	It gives a sequence of lower bounds for the minimum value $f_{min}$ of (\ref{1.1}). Under the archimedeanness, we can prove the asymptotic convergence of this unified hierarchy. Moreover, under some further local optimality conditions, we can prove its finite convergence. We, in addition, study the special properties for the univariate case. When $n = 1$, there are nice representations for polynomials that are nonnegative over intervals.
	The resulting unified Moment-SOS relaxations can be expressed in a more mathematically concise manner.
	We also present numerical experiments to demonstrate the efficiency of our unified Moment-SOS hierarchy.

	An application of \reff{1.1} is to compute the $(p,q)$-norm of a matrix $A$:
	\[
	\| A \|_{p,q} \coloneqq \max_{x \neq 0} \frac{\| Ax \|_{p}}{\|x\|_{q}} \ = \max_{\|x\|_q=1}\| Ax \|_{p},
	\]
	where $p,q$ are positive integers.
	When $p$ and $q$ are both even, this is a standard polynomial optimization problem. If one of them is odd, the norm $\| A \|_{p,q}$ can be expressed as the optimal value of a problem like \reff{1.1}. For instance, when $p=4$ and $q=3$, we can formulate this problem as
	\begin{equation}\label{norm_43}
		\left\{	
		\begin{array}{cl}
			\max & (\|Ax\|_4)^4 \\
			s.t. & |x_1|^3 + \cdots + |x_n|^3 =1.
		\end{array}
		\right.
	\end{equation}
	The feasible set of the above can be expressed in the union form as in \reff{1.1}.
	It is interesting to note that the number of sets in the union is $2^n$, so the difficulty of \reff{norm_43} increases substantially as $n$ gets larger.
	More details are given in Section~\ref{sc:eg}.

	The paper is organized as follows. Section~\ref{sc:pre} introduces the notation and some preliminary results about polynomial optimization. Section~\ref{sc:uni} gives the unified hierarchy of Moment-SOS relaxations; the asymptotic and finite convergence are proved under certain assumptions. Section~\ref{sc:n=1} studies some special properties of univariate polynomial optimization. Section~\ref{sc:eg} gives numerical experiments and applications. Section~\ref{sc:con} draws conclusions and makes some discussions for future work.

	\section{Preliminaries}\label{sc:pre}

	\subsection*{Notation}
	The symbol $\mathbb{N}$ (resp., $\mathbb{R}$) stands for the set of nonnegative integers (resp., real numbers). For an integer $m>0$, denote $[m] \coloneqq \{1, 2, \ldots, m\}$. For a scalar $t \in \mathbb{R}$, $\lceil t \rceil$ denotes the smallest integer greater than or equal to $t$, and $\lfloor t \rfloor$ denotes the largest integer less than or equal to $t$. For a polynomial $p$, $\deg(p)$ denotes its total degree and vec($p$) denotes its coefficient vector. For two vectors $a$ and $b$, the notation $a \perp b$ means they are perpendicular.
	The superscript $^T$ denotes the transpose of a matrix or vector. For a symmetric matrix $X$, $X \succeq 0$ (resp., $X \succ 0$) means that $X$ is positive semidefinite (resp., positive definite). The symbol $S^n_+$ stands for the set of all $n$-by-$n$ real symmetric positive semidefinite matrices. For two symmetric matrices $X$ and $Y$, the inequality $X\succeq Y$ (resp., $X\succ Y$) means that $X-Y\succeq0$ (resp., $X-Y\succ0$). For $x \coloneqq (x_1, \ldots, x_n)$ and a power vector $\alpha \coloneqq (\alpha_1, \ldots, \alpha_n) \in \mathbb{N}^n$, denote $|\alpha| \coloneqq \alpha_1 + \cdots + \alpha_n$ and the monomial $x^{\alpha} \coloneqq x_1^{\alpha_1}\cdots x_n^{\alpha_n}$.
	For a real number $q\geq 1$, the $q$-norm of $x$ is denoted as $\|x\|_q \coloneqq (|x_1|^q + \cdots + |x_n|^q)^{1/q}$.
	The notation
	\[\mathbb{N}_d^n \coloneqq \{\alpha \in \mathbb{N}^n : |\alpha| \leq d\}\]
	denotes the set of monomial powers with degrees at most $d$.
	The symbol $\mathbb{R}^{\mathbb{N}_d^n}$ denotes the space of all real vectors labeled by $\alpha \in \mathbb{N}_d^n$. The column vector of all monomials in $x$ and of degrees up to $d$ is denoted as
	\[[x]_d \coloneqq
	\begin{bmatrix}
		1 & x_1 & \cdots & x_n & x_1^2 & x_1x_2 & \cdots & x_n^d
	\end{bmatrix}^T.
	\]
	The notation $\mathbb{R}[x] \coloneqq \mathbb{R}[x_1, \ldots, x_n]$ stands for the ring of polynomials in $x$ with real coefficients. Let $\mathbb{R}[x]_d$ be the set of real polynomials with degrees at most $d$. Denote by $\mathscr{P}(K)$ the cone of polynomials that are nonnegative on $K$ and let
	\[
	\mathscr{P}_d(K) \coloneqq \mathscr{P}(K)\cap\mathbb{R}[x]_d.
	\]

	In the following, we review some basics of polynomial optimization. For a tuple $h \coloneqq (h_1, \ldots, h_s)$ of polynomials in $\mathbb{R}[x]$, let
	\[
	\text{Ideal}[h] \coloneqq h_1 \cdot \mathbb{R}[x] + \cdots + h_s \cdot \mathbb{R}[x].
	\]
	The $2k$th \textit{truncation} of Ideal[$h$] is
	\[
	\text{Ideal}[h]_{2k} \coloneqq h_1 \cdot \mathbb{R}[x]_{2k-\deg(h_1)} + \cdots + h_s \cdot \mathbb{R}[x]_{2k-\deg(h_s)}.
	\]
	The real variety of $h$ is
	\[
	V_{\mathbb{R}}(h) = \{x\in \re^n : h(x) = 0\}.
	\]
	A polynomial $\sig \in \mathbb{R}[x]$ is said to be a sum of squares (SOS) if there are polynomials $q_1, \ldots, q_t \in \mathbb{R}[x]$ such that $\sig= q_1^2 + \cdots + q_t^2$. The convex cone of all SOS polynomials in $x$ is denoted as $\Sigma[x]$. We refer to \cite{HKL20, LasBk15, Lau09, nie2023moment} for more details. For a tuple of polynomials $g \coloneqq (g_1, \ldots, g_t)$, its \textit{quadratic module} is (let $g_0 \coloneqq 1$)
	\[
	\text{QM}[g] \coloneqq \Big \{ \sum_{i=0}^{t} \sigma_{i} g_{i} \,\Big \vert\,
        \text{each} \, \, \sigma_{i} \in \Sigma[x] \Big \}.
	\]
	For a positive integer $k$, the degree-$2k$ truncation of QM[$g$] is
	\[
	\text{QM}[g]_{2k} \coloneqq \Big\{ \sum_{i=0}^{t} \sigma_{i} g_{i} \,\Big \vert\, \sigma_{i} \in \Sigma[x], \deg(\sigma_{i}g_i) \leq 2k \Big\}.
	\]
	The quadratic module QM[$g$] is said to be \textit{archimedean} if there exists $q \in \text{QM}[g]$ such that the set
	\[
	S(q) \coloneqq \{x \in \mathbb{R}^n \mid q(x) \geq 0\}
	\]
	is compact.

	\begin{theorem}\cite{putinar1993positive}\label{thm 2.1}
		If $\text{\normalfont QM}[g]$ is archimedean and a polynomial $f > 0$ on $S(g)$, then $f \in \text{\normalfont QM}[g]$.
	\end{theorem}

	A vector $y \coloneqq \left(y_{\alpha}\right)_{\alpha \in \mathbb{N}^n_{2k}}$ is said to be a \textit{truncated multi-sequences} (tms) of degree $2k$.
	For $y \in \mathbb{R}^{\mathbb{N}^n_{2k}}$, the \textit{Riesz functional} determined by $y$ is the linear functional $\mathscr{L}_y$ acting on $\re[x]_{2k}$ such that
	\[
	\mathscr{L}_y\Big( \sum\limits_{\alpha \in \mathbb{N}^n_{2k}} p_{\alpha}x^{\alpha} \Big)
	\coloneqq \sum\limits_{\alpha \in \mathbb{N}^n_{2k}} p_{\alpha}y_{\alpha}.
	\]
	For convenience, we denote
	\[
	\langle p, y \rangle \coloneqq \mathscr{L}_y(p), \quad p\in \re[x]_{2k}.
	\]
	The \textit{localizing matrix} and \textit{localizing vector} of $p$ generated by $y$ are respectively
	\begin{align*}
	L_p^{(k)}[y] &\coloneqq \mathscr{L}_y(p(x) \cdot [x]_{s_1}[x]_{s_1}^T), \\
	\mathscr{V}_{p}^{(2k)}[y] &\coloneqq \mathscr{L}_y(p(x) \cdot [x]_{s_2}).
\end{align*}
	In the above, the linear operator is applied component-wisely and
	\[
	s_1 \coloneqq \lceil k - \deg(p)/2 \rceil, \quad s_2 \coloneqq 2k - \deg(p).
	\]
	We remark that $L_p^{(k)}[y] \succeq 0$ if and only if $\mathscr{L}_y \ge 0$ on $\text{QM}[p]_{2k}$, and $\mathscr{V}_{p}^{(2k)}[y]=0$ if and only if $\mathscr{L}_y = 0$ on $\text{Ideal}[p]_{2k}$.
	More details for this can be found in \cite{LasBk15, Lau09, nie2023moment}.
	The localizing matrix $L_p^{(k)}[y]$ satisfies the following equation
	\[
	\left\langle p(x)\left(v^T[x]_s\right)^2, y \right\rangle = v^T\left(L_p^{(k)}[y]\right)v
	\]
	for the degree $s \coloneqq k - \lceil \deg(p)/2 \rceil$ and for every vector $v$ of length $n+s \choose s$. For instance, when $n = 3$, $k = 3$ and $p =  x_1x_2x_3 - x_3^3$,
	\[
	L_p^{(3)}[y] =
	\begin{bmatrix}
		y_{111}-y_{003} & y_{211}-y_{103} & y_{121}-y_{013} & y_{112}-y_{004} \\
		y_{211}-y_{103} & y_{311}-y_{203} & y_{221}-y_{113} & y_{212}-y_{104} \\
		y_{121}-y_{013} & y_{221}-y_{113} & y_{131}-y_{023} & y_{122}-y_{014} \\
		y_{112}-y_{004} & y_{212}-y_{104} & y_{122}-y_{014} & y_{113}-y_{005}
	\end{bmatrix}.
	\]
	In particular, for $p=1$, we get the \textit{moment matrix} $M_k[y] \coloneqq L_1^{(k)}[y]$.
	Similarly, the localizing vector $\mathscr{V}_{p}^{(2k)}[y]$ satisfies
	\[
	\left\langle p(x)\left(v^T[x]_t\right), y \right\rangle = \Big(\mathscr{V}_{p}^{(2k)}[y]\Big)^Tv
	\]
	for $t \coloneqq 2k - \deg(p)$.
	For instance, when $n = 3$, $k = 2$ and $p =  x_1^2 + x_2^2 + x_3^2 - 1$,
	\[
	\mathscr{V}_{p}^{(4)}[y] =
	\begin{bmatrix}
		y_{200} + y_{020}  + y_{002} - y_{000} \\
		y_{300} + y_{120} + y_{102} - y_{100} \\
		y_{210} + y_{030} + y_{012} - y_{010} \\
		y_{201} + y_{021} + y_{003} - y_{001} \\
		y_{400} + y_{220} + y_{202} - y_{200} \\
		y_{310} + y_{130} + y_{112} - y_{110} \\
		y_{301} + y_{121} + y_{103} - y_{101} \\
		y_{220} + y_{040} + y_{022} - y_{020} \\
		y_{211} + y_{031} + y_{013} - y_{011} \\
		y_{202} + y_{022} + y_{004} - y_{002}
	\end{bmatrix}.
	\]
	It is worthy to note that if $L_{g_i}^{(k)}[y] \succeq 0$ and $f \in \text{\normalfont QM}[g]_{2k}$, then $\langle f \; , \; y \rangle \geq 0$. This can be seen as follows. For $f = \sum\limits_{i=0}^{t} g_i\sigma_{i}$ with $\sigma_{i} = \sum\limits_j p_{ij}^2 \in \Sigma[x]$ and $\deg( g_i\sigma_{i}) \leq 2k$, we have
	\[
	\langle f \; , \; y \rangle = \Big\langle \sum\limits_{i=0}^{t} g_i\sigma_{i} \; , \; y \Big\rangle
	= \sum\limits_{i, j} \text{vec}(p_{ij})^T \big(L_{g_i}^{(k)}[y] \big) \text{vec}(p_{ij}) \geq 0.
	\]

	A tms $y\in \re^{\N_{2k}^n}$ is said to \textit{admit} a Borel measure $\mu$ if
	\[
	y_{\alpha} = \int x^{\alpha} d\mu \quad \text{for all} ~ \alpha \in \N^n_{2k}.
	\]
	Such $\mu$ is called a \textit{representing measure} for $y$. The \textit{support} of $\mu$ is the smallest closed set $S \subseteq \mathbb{R}^n$ such that $\mu(\mathbb{R}^n \setminus S) = 0$, denoted as $supp(\mu)$. The measure $\mu$ is said to be supported in a set $K$ if $supp(\mu) \subseteq K$.

	\subsection{Moment relaxation}\label{sub 2.1}
	Consider the polynomial optimization problem
	\begin{equation}\label{1.2}
		\left\{ \baray{cl}
			\min & f(x)\\
			s.t. &  c_i(x) = 0 ~ (i \in \mathcal{E}), \\
			& c_j(x) \geq 0 ~ (j \in \mathcal{I}),
		\earay \right.
	\end{equation}
	where $f, c_i, c_j$ are polynomials in $x$.
	The $k$th order moment relaxation for (\ref{1.2}) is
	\begin{equation}\notag
		\left\{ \baray{cl}
			\min & \langle f \; , \; y\rangle \\
			s.t. & \mathscr{V}_{c_i}^{(2k)}[y] = 0 ~ (i \in \mathcal{E}), \\
			& L_{c_j}^{(k)}[y] \succeq 0 ~ (j \in \mathcal{I}), \\
			& M_k[y] \succeq 0, \\
			& y_0 = 1, y \in \mathbb{R}^{\mathbb{N}^n_{2k}}.
		\earay \right.
	\end{equation}
	Suppose the tms $y^*$ is a minimizer of above.
	Denote the degree
	\[
	d \coloneqq \max_{i \in \mathcal{E} \cup \mathcal{I}} \{ \lceil \deg(c_{i})/2 \rceil \}.
	\]
	We can extract minimizers if $y^*$ satisfies the \textit{flat truncation} condition: there exists an integer $k \ge \, t \geq \max\{d, \deg(f)/2\}$ such that
	\begin{equation}\label{FT:y*}
		\text{rank}\,M_{t-d}[y^{*}] = \text{rank}\,M_t[y^{*}].
	\end{equation}
	Interestingly, if (\ref{FT:y*}) holds, we can extract $r \coloneqq \text{rank}\,M_t[y^{*}]$ minimizers for the optimization problem (\ref{1.2}).
	
The following result is based on work by Curto and Fialkow \cite{curto2005truncated} and Henrion and Lasserre \cite{henrion2005detecting}. The form of the result as presented here can be found in book \cite[Section~2.7]{nie2023moment}.
	\begin{theorem} \cite{curto2005truncated,henrion2005detecting}
		\label{FT}
		If $y^*$ satisfies (\ref{FT:y*}), then there exist  $r \coloneqq \text{\normalfont rank}\,M_t[y^{*}]$ distinct feasible points $u_1, \ldots, u_r$ for (\ref{1.2}) and positive scalars $\lambda_1, \ldots, \lambda_r$ such that
		\[
		y^{*}|_{2t} = \lambda_1[u_1]_{2t} + \cdots + \lambda_r[u_r]_{2t}.
		\]
		In the above, the notation $y^{*}|_{2t}$ stands for its subvector of entries that are labeled by $\alpha \in \mathbb{N}_{2t}^n$.
	\end{theorem}

	\subsection{Optimality conditions}
	\label{ssc:opcd}

	Suppose  $u$ is a local minimizer of (\ref{1.2}). Denote the active labeling set
	\[
	J(u) \coloneqq \{j \in \mathcal{I} : c_j(u) = 0\}.
	\]
	The  \textit{linear independence constraint qualification condition} (LICQC) holds at $u$ if the gradient set $\{\nabla c_i(u)\}_{i \in \mathcal{E} \cup J(u)}$ is linearly independent. When LICQC holds, there exists a Lagrange multiplier vector
	\[
	\lambda \coloneqq (\lambda_i)_{i \in \mathcal{E}} \cup (\lambda_j)_{j \in \mathcal{I}}
	\]
	satisfying the Karush-Kuhn-Tucker (KKT) conditions
	\begin{equation}\label{2.1}
		\nabla f(u) = \sum\limits_{i \in \mathcal{E}}\lambda_i \nabla c_i(u) + \sum\limits_{j \in \mathcal{I}}\lambda_j \nabla c_j(u),
	\end{equation}
	\begin{equation}\label{2.2}
		0 \leq c_j(u) \perp \lambda_j \geq 0, \quad  \text{for all } j \in \mathcal{I}.
	\end{equation}
	The equation (\ref{2.1}) is known as the \textit{first order optimality condition} (FOOC), and (\ref{2.2}) is called the \textit{complementarity condition} (CC).  If, in addition, $\lambda_j + c_j(u) > 0$ for all $j \in \mathcal{I}$, the \textit{strict complementarity condition} (SCC) is said to hold at $u$. For the $\lambda_i$ satisfying (\ref{2.1})-(\ref{2.2}), the Lagrange function is
	\[
	L(x) \coloneqq f(x) - \sum\limits_{i \in \mathcal{E}}\lambda_i c_i(x) - \sum\limits_{j \in \mathcal{I}}\lambda_j c_j(x).
	\]
	The Hessian of the Lagrangian is
	\[
	\nabla^2L(x) \coloneqq \nabla^2f(x) - \sum\limits_{i \in \mathcal{E}}\lambda_i \nabla^2c_i(x) - \sum\limits_{j \in \mathcal{I}}\lambda_j \nabla^2c_j(x).
	\]
	If $u$ is a local minimizer and LICQC holds, the \textit{second order necessary condition} (SONC) holds at $u$:
	\begin{equation}\notag
		v^T\Bigl(\nabla^2L(u)\Bigr)v \geq 0 \quad \text{for all} \quad v \in \bigcap_{i \in \mathcal{E} \cup J(u)} \nabla c_i(u)^{\perp},
	\end{equation}
	where $\nabla c_i(u)^{\perp} \coloneqq \{v \in \mathbb{R}^n \mid \nabla c_i(u)^Tv = 0\}$.
	Stronger than SONC is the \textit{second order sufficient condition} (SOSC):
	\begin{equation}\notag
		v^T\Bigl(\nabla^2L(u)\Bigr)v > 0 \quad \text{for all} \quad 0 \neq v \in \bigcap_{i \in \mathcal{E} \cup J(u)} \nabla c_i(u)^{\perp}.
	\end{equation}
	If a feasible point $u$ satisfies FOOC, SCC, and SOSC, then $u$ must be a strict local minimizer. We refer to the book \cite{Brks} for optimality conditions in nonlinear programming.

	\section{A Unified Moment-SOS Hierarchy}\label{sc:uni}
	In this section, we give a unified hierarchy of Moment-SOS relaxations to solve (\ref{1.1}). Under some assumptions, we prove this hierarchy has asymptotic or finite convergence.

	\subsection{Unified Moment-SOS relaxations}
	For convenience of description, we denote the equality and inequality constraining polynomial tuples for $K_l$ as
	\[
	c^{(l)}_{eq} \coloneqq (c^{(l)}_i)_{i \in \mathcal{E}^{(l)}}, \quad c^{(l)}_{in} \coloneqq (c^{(l)}_j)_{j \in \mathcal{I}^{(l)}}.
	\]
	Recall that Ideal[$c^{(l)}_{eq}$] denotes the ideal generated by $c^{(l)}_{eq}$ and QM[$c^{(l)}_{in}$] denotes the quadratic module generated by $c^{(l)}_{in}$.  We refer to Section~\ref{sc:pre} for the notation. The minimum value of (\ref{1.1}) is denoted as $f_{min}$ and its feasible set is $K$.
	We look for the largest scalar $\gamma$ that is a lower bound of $f$ over $K$, i.e., $f - \gamma \in \mathscr{P}(K)$. Since
	\[
	K = K_1 \cup K_2 \cup \cdots \cup K_m,
	\]
	we have $f - \gamma \geq 0$ on $K$ if and only if $f - \gamma \geq 0$ on $K_l$ for every $l = 1, \ldots, m$. Note that $f - \gamma \geq 0$ on $K_l$ is ensured by the membership (for some degree $2k$)
	\[
	f - \gamma \in \text{Ideal}[c^{(l)}_{eq}]_{2k} + \text{QM}[c^{(l)}_{in}]_{2k}.
	\]
	The $k$th order SOS relaxation for solving (\ref{1.1}) is therefore
	\begin{equation}\label{3.1}
	 \left\{ \baray{cl}
			\max & \gamma \\
			s.t. & f - \gamma \in \bigcap\limits_{l=1}^{m} \left[\text{Ideal}[c^{(l)}_{eq}]_{2k} + \text{QM}[c^{(l)}_{in}]_{2k}\right].
		\earay \right.
	\end{equation}
	The dual optimization of \reff{3.1} is then the moment relaxation
	\begin{equation}\label{3.2}
		\left\{ \baray{cl}
			\min & \langle f \; , \; y^{(1)}\rangle + \cdots + \langle f \; , \; y^{(m)}\rangle \\
			s.t. & \mathscr{V}_{c^{(l)}_i}^{(2k)}[y^{(l)}] = 0 ~ (i \in \mathcal{E}^{(l)}), \\
			& L_{c^{(l)}_j}^{(k)}[y^{(l)}] \succeq 0 ~ (j \in \mathcal{I}^{(l)}), \\
			& M_k[y^{(l)}] \succeq 0, \sum\limits_{l=1}^{m}y_0^{(l)}=1, \\
			&  y^{(l)} \in \mathbb{R}^{\mathbb{N}^n_{2k}},  l = 1, \ldots, m. \\
		\earay \right.
	\end{equation}
	The integer $k$ is called the \textit{relaxation order}. For $k = 1,2, \ldots$, the sequence of primal-dual pairs (\ref{3.1})-(\ref{3.2}) is called the \textit{unified} Moment-SOS hierarchy. For each $k$, we denote by $f_{sos, k}$ and $f_{mom, k}$ the optimal values of (\ref{3.1}) and (\ref{3.2}) respectively.
We remark that the moment relaxation \reff{3.2} can be equivalently written in terms of Riesz functional.
Let $\mathscr{L}^{(l)}$ denote the Riesz functional given by $y^{(l)}$, then \reff{3.2} is equivalent to
\begin{equation}\notag
\left\{ \baray{cl}
\min & \mathscr{L}^{(1)}(f) + \cdots + \mathscr{L}^{(m)}(f) \\
s.t. & \mathscr{L}^{(l)} = 0 \text{ on Ideal}[c^{(l)}_{eq}]_{2k}, \\
      & \mathscr{L}^{(l)} \ge 0 \text{ on QM}[c^{(l)}_{in}]_{2k}, \\
      & \mathscr{L}^{(1)}(1) + \cdots + \mathscr{L}^{(m)}(1)= 1, \\
      & l = 1, \ldots, m.
\earay \right.
\end{equation}

	\begin{prop}\label{prop 3.2}
		For each relaxation order $k$, it holds that
		\begin{equation}\label{bound}
			f_{sos,k} \leq f_{mom,k} \leq f_{min}.
		\end{equation}
		Moreover, both sequences $\{f_{sos,k}\}_{k=1}^{\infty}$ and $\{f_{mom,k}\}_{k=1}^{\infty}$ are monotonically increasing.
	\end{prop}
	\begin{proof}
		By the weak duality, we have $f_{sos,k} \leq f_{mom,k}$. For every $\epsilon>0$, there exist $l' \in [m]$ and $u \in K_{l'}$ such that $f(u) \leq f_{min}+\epsilon$. Let $y \coloneqq (y^{(1)}, \ldots, y^{(m)})$ be such that $y^{(l')} = [u]_{2k}$ and $y^{(l)} = 0$ for all $l \in [m] \setminus \{l'\}$. Then, $y$ is feasible for (\ref{3.2}) and
		\[
		f_{mom,k} \leq \langle f\; , \; y^{(1)}\rangle + \cdots + \langle f\; , \; y^{(m)}\rangle = \langle f\; , \; y^{(l')}\rangle = f(u) \leq f_{min} + \epsilon.
		\]
		Since $\epsilon > 0$ can be arbitrary, $f_{mom,k} \leq f_{min}$. Therefore, we get \reff{bound}.
		Clearly, if $\gamma$ is feasible for (\ref{3.1}) with an order $k$, then $\gamma$ must also be feasible for (\ref{3.1}) with all larger values of $k$, since the feasible set gets larger as $k$ increases. So the sequence of lower bounds $\{f_{sos,k}\}_{k=1}^{\infty}$ is monotonically increasing. On the other hand, when $k$ increases, the feasible set of (\ref{3.2}) shrinks, so the minimum value of (\ref{3.2}) increases. Therefore, $\{f_{mom,k}\}_{k=1}^{\infty}$ is also monotonically increasing.
	\end{proof}

	\subsection{Extraction of minimizers}
	We show how to extract minimizers of (\ref{1.1}) from the unified moment relaxation.
	This is a natural extension from the case $m=1$ in Section~\ref{sub 2.1}.
	Suppose the tuple $y^* \coloneqq (y^{(*,1)}, \ldots, y^{(*,m)})$ is a minimizer of (\ref{3.2}). Denote the degree
	\[
	d_{l} \coloneqq \max_{i \in \mathcal{E}^{(l)} \cup \mathcal{I}^{(l)}} \{ \lceil \deg(c^{(l)}_{i})/2 \rceil \}.
	\]
	We can extract minimizers by checking the \textit{flat truncation} condition: there exists an integer $t \geq \max\limits_{l \in [m]}\{d_{l}, \deg(f)/2\}$ such that
	\begin{equation}\label{3.3}
		\text{rank}\,M_{t-d_l}[y^{(*,l)}] = \text{rank}\,M_t[y^{(*,l)}] \quad \text{for each} ~ l\in \mathcal{A},
	\end{equation}
	where the labeling set	
	\[
	\mathcal{A} \coloneqq \big\{l \in [m] : y_0^{(l)}>0\big\}.
	\]
	Interestingly, if (\ref{3.3}) holds, we can extract
	\be\label{rank r}
	r \coloneqq \sum_{l \in \mc{A}} \text{rank}\,M_t[y^{(*,l)}]
	\ee
	minimizers for the optimization problem (\ref{1.1}).

	\begin{alg}\label{alg}\rm
		To solve the polynomial optimization (\ref{1.1}), do the following:
		\begin{itemize}
			\item[Step~0] Let $k \coloneqq \max\limits_{l \in [m]}\{d_{l}, \lceil \deg(f)/2\rceil\}$.
			
			\item[Step~1] Solve the relaxation \reff{3.2}.
			If it is infeasible, output that (\ref{1.1}) is infeasible and stop. Otherwise, solve it for a minimizer $y^* \coloneqq (y^{(*,1)}, \ldots, y^{(*,m)})$.
			
			\item[Step~2] Check if the flat truncation (\ref{3.3}) holds or not. If \reff{3.3} holds, then the relaxation \reff{3.2} is tight and for each $l \in \mathcal{A}$, the truncation $y^{(*,l)}|_{2t}$ admits a finitely atomic measure $\mu^{(l)}$ such that each point in $supp(\mu^{(l)})$ is a minimizer of (\ref{1.1}). Moreover, $f_{min} = f_{mom,k}$.
			
			\item [Step~3] If \reff{3.3} fails, let $k \coloneqq k+1$ and go to Step~1.
		\end{itemize}
	\end{alg}

	The conclusion in Step~2 is justified by the following.

	\begin{theorem}\label{lem 3.5}
		Let $y^* \coloneqq (y^{(*,1)}, \ldots, y^{(*,m)})$ be a minimizer of (\ref{3.2}). Suppose (\ref{3.3}) holds for all $l\in\mathcal{A}$. Then, the moment relaxation (\ref{3.2}) is tight and for each $l \in \mc{A}$, the truncation
		\[
		y^{(*,l)}|_{2t} \coloneqq (y^{(*,l)}_{\alpha})_{\alpha \in \mathbb{N}^n_{2t}}
		\]
		admits a $r_l$-atomic measure $\mu^{(l)}$, where $r_l = \text{\normalfont rank}\, M_t[y^{(*,l)}]$, and each point in $supp(\mu^{(l)})$ is a minimizer of (\ref{1.1}). Therefore, the total number of minimizers is $r$ as in \reff{rank r}.
	\end{theorem}
	\begin{proof}
		By the assumption, the $y^{(*,l)} \in \mathbb{R}^{\mathbb{N}^n_{2k}}$ satisfies (\ref{3.3}) and
		\[
		L_{c^{(l)}_j}^{(k)}[y^{(*,l)}] \succeq 0 ~ (j \in \mathcal{I}^{(l)}), \quad M_k[y^{(*,l)}] \succeq 0.
		\]
		Then, by Theorem \ref{FT}, there exists $r_l$ distinct points $u_1^{(l)}, \ldots, u_{r_l}^{(l)} \in K_l$ and positive scalars $\lambda_1^{(l)}, \ldots, \lambda_{r_l}^{(l)}$ such that
		\[
		y^{(*,l)}|_{2t} = \lambda_1^{(l)}[u_1^{(l)}]_{2t} + \cdots + \lambda_{r_l}^{(l)}[u_{r_l}^{(l)}]_{2t}.
		\]
		The constriant $\sum\limits_{l=1}^{m} y_0^{(l)} = 1$ implies that $\sum\limits_{l=1}^{m} \sum\limits_{i=1}^{r_l} \lambda_i^{(l)} = 1$, so
		\[
		\sum\limits_{l=1}^{m} \sum\limits_{i=1}^{r_l} \lambda_i^{(l)}f(u_i^{(l)}) = \sum\limits_{l=1}^{m} \langle f, y^{(*, l)}|_{2t} \rangle = \sum\limits_{l=1}^{m} \langle f, y^{(*, l)} \rangle = f_{mom,k} \leq f_{min}.
		\]
		For each $u_i^{(l)} \in K_l$, we have $f(u_i^{(l)})  \geq f_{min}$, so
		\[
		\sum\limits_{l=1}^{m} \sum\limits_{i=1}^{r_l} \lambda_i^{(l)}f(u_i^{(l)}) \geq \sum\limits_{l=1}^{m} \sum\limits_{i=1}^{r_l} \lambda_i^{(l)} f_{min} = f_{min}.
		\]
		Hence, $f_{mom,k} = f_{min}$ and
		 \[
		 \sum\limits_{l=1}^{m} \sum\limits_{i=1}^{r_l} \lambda_i^{(l)} \left[ f(u_i^{(l)}) - f_{min} \right] = 0.
		 \]
		 Since each $\lambda_i^{(l)} > 0$, then each $f(u_i^{(l)})  = f_{min}$, i.e., each $u_i^{(l)}$ is a minimizer of (\ref{1.1}).
	\end{proof}
	
	In Step~2, the flat truncation condition (\ref{3.3}) is used to extract minimizers. When it holds, a numerical method is given in \cite{henrion2005detecting} for computing the minimizers. We refer to \cite[Section~2.7]{nie2023moment} for more details.

	\subsection{Convergence analysis}

	Recall that $f_{min}, f_{sos,k}$ and $f_{mom,k}$ denote the optimal values of (\ref{1.1}), (\ref{3.1}) and (\ref{3.2}) respectively. The unified Moment-SOS hierarchy (\ref{3.1})-(\ref{3.2}) is said to have \textit{asymptotic convergence} if $f_{sos,k} \to f_{min}$ as $k \to \infty$. If $f_{sos,k} = f_{min}$ for some $k$, this unified hierarchy is said to be \textit{tight} or to have \textit{finite convergence}.
	The following theorem is a natural extension from the case $m=1$.

	\begin{theorem}[Asymptotic convergence]\label{thm 3.6}
		If $\text{\normalfont Ideal}[c_{eq}^{(l)}] + \text{\normalfont QM}[c_{in}^{(l)}]$ is archimedean for every $l = 1, \ldots, m$, then the Moment-SOS hierarchy (\ref{3.1})-(\ref{3.2}) has asymptotic convergence:
		\[
		\lim_{k \to \infty}f_{sos,k} = \lim_{k \to \infty}f_{mom,k} = f_{min}.
		\]
	\end{theorem}
	\begin{proof}
		For $\epsilon>0$, let $\gamma = f_{min} - \epsilon$. Then
		\[
		f(x) - \gamma = f(x) - f_{min} + \epsilon > 0
		\]
		 on $K_l$. Since $\text{\normalfont Ideal}[c_{eq}^{(l)}] + \text{\normalfont QM}[c_{in}^{(l)}]$ is archimedean for every $l$, by Theorem \ref{thm 2.1},
		 \[
		 f(x) - \gamma \in \text{\normalfont Ideal}[c_{eq}^{(l)}]_{2k} + \text{\normalfont QM}[c_{in}^{(l)}]_{2k}
		 \]
		for all $k$ large enough. So
		\[
		f_{min} - \epsilon = \gamma \leq f_{sos,k} \leq f_{min} \implies f_{min} - \epsilon \leq \lim\limits_{k \to \infty}f_{sos,k} \leq f_{min}.
		\]
		Since $\epsilon > 0$ can be arbitrary, $\lim\limits_{k \to \infty}f_{sos,k} = f_{min}$. By \reff{bound}, we get the desired conclusion.
	\end{proof}

	Recall the linear independence constraint qualification condition (LICQC), the strict complementarity condition (SCC), and the second order sufficient condition (SOSC) introduced in Subsection \ref{ssc:opcd}. The following is the conclusion for the finite convergence of the unified Moment-SOS hierarchy of (\ref{3.1})-(\ref{3.2}).

	\begin{theorem}[Finite convergence]\label{thm 3.7}
		Assume $\text{\normalfont Ideal}[c_{eq}^{(l)}] + \text{\normalfont QM}[c_{in}^{(l)}]$ is archimedean for every $l = 1, \ldots, m$. If the LICQC, SCC, and SOSC hold at every global minimizer of (\ref{1.1}) for each $K_l$, then the Moment-SOS hierarchy (\ref{3.1})-(\ref{3.2}) has finite convergence, i.e.,
		\[
		f_{sos,k} = f_{mom, k} = f_{min}
		\]
		for all $k$ large enough.
	\end{theorem}
	\begin{proof}
		We denote by $f_{min,l}$ the minimum value of $f$ on the set $K_l$. Let
		\[
		\mathcal{B} \coloneqq \{l : f_{min,l} = f_{min}\}.
		\]
		\begin{enumerate}[(i)]
			\item For the case $l \notin \mathcal{B}$, $f_{min,l} > f_{min}$,
			\[
			f(x) - f_{min} \geq f_{min,l} - f_{min} > 0
			\]
			on $K_l$. Since $\text{\normalfont Ideal}[c_{eq}^{(l)}] + \text{\normalfont QM}[c_{in}^{(l)}]$ is archimedean, there exists $k_0$ such that
			\[
			f- (f_{min} - \epsilon) \in \text{\normalfont Ideal}[c_{eq}^{(l)}]_{2k_0} + \text{\normalfont QM}[c_{in}^{(l)}]_{2k_0}
			\]
			for all $\epsilon>0$.
			\item For the case $l \in \mathcal{B}$, $f_{min,l} = f_{min}$.
			Since the LICQC, SCC, and SOSC hold at every global minimizer $x^*$ of (\ref{1.1}),
			there exists a degree $k_0$ such that for all $\eps>0$, we have
			\[
			f - (f_{min} - \epsilon) \in \text{\normalfont Ideal}[c_{eq}^{(l)}]_{2k_0} + \text{\normalfont QM}[c_{in}^{(l)}]_{2k_0}.
			\]
			This is shown in the proof of Theorem~1.1 of \cite{NieOPCD}.
		\end{enumerate}
			Combining cases (i) and (ii), we know that $\gamma = f_{min} - \epsilon$ is feasible for (\ref{3.1}) with the order $k_0$. Hence, $f_{sos,k_0} \geq \gamma = f_{min} - \epsilon$. Since $\epsilon > 0$ can be arbitrary, we get $f_{sos,k_0} \geq f_{min}$. By Proposition \ref{prop 3.2}, we get $f_{sos,k} = f_{mom, k} = f_{min}$ for all $k \geq k_0$.
	\end{proof}

	\section{Univariate Polynomial Optimization}\label{sc:n=1}
	In this section, we consider the special case of univariate polynomial optimization, i.e., $n=1$.
The following results for the univariate case are extensions from the single interval case, and are presented here to provide a complete and thorough understanding for convenience of readers.
	The problem (\ref{1.1}) can be expressed as
	\begin{equation}\label{5.1}
		\left\{ \baray{cl}
			\min & f(x) \coloneqq f_0 + f_1x + \cdots + f_{d}x^{d}\\
			s.t. & x \in \bigcup\limits_{l=1}^m K_l,
		\earay \right.
	\end{equation}
	where $K_l= [a_l,b_l]$ with $a_l<b_l$ for $l = 1, \ldots, m$. We still denote by $f_{min}$ the minimum value of (\ref{5.1}). For convenience, we only consider compact intervals. The discussions for unbounded intervals like $(-\infty, b]$ and $[a, +\infty)$ are similar (see \cite[Chapter~3]{nie2023moment}).

	Let $y \coloneqq (y_0, \ldots, y_d) \in \mathbb{R}^{d+1}$ be a univariate tms of degree $d$ with $d=2d_0+1$ or $d=2d_0$. The $d_0$th order moment matrix of $y$ is
	\[
	M_{d_0}[y] =
	\begin{bmatrix}
		y_0 & y_1 & \cdots & y_{d_0} \\
		y_1 & y_2 & \cdots & y_{d_0+1} \\
		\vdots & \vdots & \ddots & \vdots \\
		y_{d_0} & y_{d_0+1} & \cdots & y_{2d_0} \\
	\end{bmatrix}.
	\]
	For convenience of notation, we also denote that
	\[
	G_{d_0}[y] \coloneqq
	\begin{bmatrix}
		y_2 & y_3 & \cdots & y_{d_0+1} \\
		y_3 & y_4 & \cdots & y_{d_0+2} \\
		\vdots & \vdots & \ddots & \vdots \\
		y_{d_0+1} & y_{d_0+2} & \cdots & y_{2d_0} \\
	\end{bmatrix},
	\]
	\[
	N_{d_0}[y] \coloneqq
	\begin{bmatrix}
		y_1 & y_2 & \cdots & y_{d_0+1} \\
		y_2 & y_3 & \cdots & y_{d_0+2} \\
		\vdots & \vdots & \ddots & \vdots \\
		y_{d_0+1} & y_{d_0+2} & \cdots & y_{2d_0+1} \\
	\end{bmatrix}.
	\]
	
	It is well-known that polynomials that are nonnegative over an interval can be expressed in terms of sum of squares. The following results were known to Lukács \cite{lukacs1918verscharfung}, Markov \cite{markov1948lectures}, Pólya and Szegö \cite{polya1972problems}, Powers and Reznick \cite{powers2000polynomials}.
	For each $h \in \re[x]_d$ that is nonnegative on the interval $[a_l, b_l]$, we have:
		\begin{enumerate}[(i)]
		\item If $d=2d_0+1$ is odd, then there exist $p,q \in \mathbb{R}[x]_{d_0}$ such that
		\be\label{sos odd}
			h=(x-a_l)p^{2}+(b_l-x)q^{2}.
		\ee
		\item If $d=2d_0$ is even, then there exist $p \in \mathbb{R}[x]_{d_0}, q \in \mathbb{R}[x]_{d_0-1}$ such that
		\be\label{sos even}
			h=p^{2}+(x-a_l)(b_l-x)q^{2}.
		\ee
		\end{enumerate}
		The optimization problem (\ref{5.1}) can be solved by the unified Moment-SOS hierarchy of \reff{3.1}-\reff{3.2}. For the univariate case, they can be simplified. We discuss in two different cases of $d$.

	\subsection{The case $d$ is odd ($d=2d_0+1$)}
	When the degree $d=2d_0+1$ is odd, by the representation \reff{sos odd}, $f_{min}$ equals the maximum value of the SOS relaxation
	\begin{equation}\label{5.2}
		\left\{ \baray{cl}
			\max & \gamma \\
			s.t. & f-\gamma=(x-a_l)[x]_{d_0}^{T}X_{l}[x]_{d_0} + (b_l-x)[x]_{d_0}^{T}Y_{l}[x]_{d_0}, \\
			& X_{l} \in S_{+}^{d_0+1}, Y_{l} \in S_{+}^{d_0+1}, l = 1, \ldots, m.
		\earay \right.
	\end{equation}
	Its dual optimization is the moment relaxation
	\begin{equation}\label{5.3}
		\left\{ \baray{cl}
			\min & \langle f\; , \; y^{(1)}\rangle + \cdots + \langle f\; , \; y^{(m)}\rangle \\
			s.t. & y_0^{(1)} + \cdots + y_0^{(m)} = 1, \\
			& b_lM_{d_0}[y^{(l)}] \succeq N_{d_0}[y^{(l)}] \succeq a_lM_{d_0}[y^{(l)}], \\
			& y^{(l)}=(y_{0}^{(l)}, y_{1}^{(l)}, \ldots, y_{2d_0+1}^{(l)}), \\
			&  l = 1, \ldots, m.
		\earay \right.
	\end{equation}
	In the above,
	\[
	\langle f\; , \; y^{(l)}\rangle = f_0y_0^{(l)} + \cdots + f_{2d_0+1}y_{2d_0+1}^{(l)}.
	\]
	Denote by $f_{sos}$ and $f_{mom}$ the optimal values of (\ref{5.2}) and (\ref{5.3}) respectively. For all $(X_l, Y_l)$ that is feasible for (\ref{5.2}) and for all $y^{(l)}$ that is feasible for (\ref{5.3}), we have
	\[
	\langle f\; , \; y^{(1)}\rangle + \cdots + \langle f\; , \; y^{(m)}\rangle \geq \gamma.
	\]
	This is because
	\[
	b_lM_{d_0}[y^{(l)}] - N_{d_0}[y^{(l)}] = L_{b_l-x}^{(d_0+1)}[y^{(l)}] \succeq 0,
	\]
	\[
	N_{d_0}[y^{(l)}] - a_lM_{d_0}[y^{(l)}] = L_{x-a_l}^{(d_0+1)}[y^{(l)}] \succeq 0,
	\]
	which implies that
	\begin{align*}
		&~ \quad  \langle f\; , \; y^{(1)}\rangle + \cdots + \langle f\; , \; y^{(m)}\rangle - \gamma \\
		&= \langle f-\gamma\; , \; y^{(1)}\rangle + \cdots + \langle f-\gamma\; , \; y^{(m)}\rangle \\
		&= \sum\limits_{l=1}^m \left[ \left\langle L_{x-a_l}^{(d_0+1)}[y^{(l)}]\; , \; X_l \right\rangle + \left\langle L_{b_l-x}^{(d_0+1)}[y^{(l)}]\; , \; Y_l \right\rangle \right] \\
		&\geq 0.
	\end{align*}
	\noindent Indeed, we have the following theorem.

	\begin{theorem}\label{thm 5.3}
		For the relaxations (\ref{5.2}) and (\ref{5.3}), we always have
		\[
		f_{sos} = f_{mom} = f_{min}.
		\]
	\end{theorem}
	\begin{proof}
		By the representation \reff{sos odd}, for $\gm=f_{min}$, the subtraction $f - f_{min}$ can be represented as in \reff{5.2} for each $l = 1, \ldots, m$, so $f_{sos} = f_{min}$. By the weak duality, we have $f_{sos} \leq f_{mom} \leq f_{min}$. Hence, they are all equal.
	\end{proof}
	
	The optimizers for \reff{5.1} can be obtained by the following algorithm.
	
	\begin{alg}\cite[Algorithm 3.3.6]{nie2023moment}
		\label{alg 5.4}\rm
		Assume $d=2d_0+1$ and
		$(y^{(1)}, \ldots, y^{(m)})$ is a minimizer for the moment relaxation (\ref{5.3}).
		For each $l$ with $y_0^{(l)} > 0$ and $r=\text{\normalfont rank}\,M_{d_0}[y^{(l)}]$, do the following:
		\begin{itemize}
			\item[Step~1] Solve the linear system
			$$\begin{bmatrix}
				y_0^{(l)} & y_1^{(l)} & \cdots & y_{r-1}^{(l)} \\
				y_1^{(l)} & y_2^{(l)} & \cdots & y_r^{(l)} \\
				\vdots & \vdots & \ddots & \vdots \\
				y_{2d_0-r+1}^{(l)} & y_{2d_0-r+2}^{(l)} & \cdots & y_{2d_0}^{(l)}\\
			\end{bmatrix}
			\begin{bmatrix}
				g_0^{(l)}  \\
				g_1^{(l)} \\
				\vdots \\
				g_{r-1}^{(l)} \\
			\end{bmatrix} =
			\begin{bmatrix}
				y_r^{(l)}  \\
				y_{r+1}^{(l)} \\
				\vdots \\
				y_{2d_0+1}^{(l)} \\
			\end{bmatrix}.$$
			\item[Step~2] Compute $r$ distinct roots $t_1^{(l)}, \ldots, t_r^{(l)}$ of the polynomial
			\[
			g^{(l)}(x) \coloneqq g_0^{(l)}+g_1^{(l)}x+\cdots+g_{r-1}^{(l)}x^{r-1}-x^r.
			\]
			\item[Step~3] The roots $t_1^{(l)}, \ldots, t_r^{(l)}$ are minimizers of the optimization problem (\ref{5.1}).
		\end{itemize}
	\end{alg}
	The conclusion in Step~3 is justified by Theorem \ref{thm 5.8}.
	The following is an exposition for the above algorithm.
	\begin{eg}\rm
		Consider the constrained optimization problem
		\begin{equation}\notag
			\left\{ \baray{cl}
				\min & x + 2x^6 - x^7 \\
				s.t. & x \in [-2, -1] \cup [1,2].
			\earay \right.
		\end{equation}
		The moment relaxation is
		\begin{equation}\notag
			\left\{ \baray{cl}
				\min & \langle f\; , \; y^{(1)}\rangle + \langle f\; , \; y^{(2)}\rangle \\
				s.t. & -M_{3}[y^{(1)}] \succeq N_{3}[y^{(1)}] \succeq -2M_{3}[y^{(1)}], \\
					  & 2M_{3}[y^{(2)}] \succeq N_{3}[y^{(2)}] \succeq M_{3}[y^{(2)}], \\
					  & y_0^{(1)} + y_0^{(2)} =1.
			\earay \right.
		\end{equation}
		The minimizer $y^* = (y^{(*,1)}, y^{(*,2)})$ of the above is obtained as
		\begin{align*}
			y^{(*,1)} =&~ 0.4191 \cdot (1, -1, 1, -1, 1, -1, 1, -1), \\
			y^{(*,2)} =&~ 0.5809 \cdot (1, 0, -2, -6, -14, -30, -62, -126)~+ \\
							&~ 0.6058 \cdot (0, 1, 3, 7, 15, 31, 63, 127).
		\end{align*}
		Applying Algorithm \ref{alg 5.4}, we get $g_0^{(1)} = -1, g_0^{(2)} = -2, g_1^{(2)} = 3$ and the polynomials
		\[
		g^{(1)}(x) = -1-x, \quad g^{(2)}(x) = -2 + 3x - x^2.
		\]
		Therefore, the minimizers are the distinct roots $-1, 1, 2$ and the global minimum value $f_{min} = 2$.
	\end{eg}

	\subsection{The case $d$ is even ($d=2d_0$)}
	When the degree $d=2d_0$ is even, by the representation \reff{sos even}, $f_{min}$ equals the maximum value of
	\begin{equation}\label{5.4}
		\left\{ \baray{cl}
			\max & \gamma \\
			s.t. & f-\gamma=[x]_{d_0}^{T}X_{l}[x]_{d_0} + (x-a_l)(b_l-x)([x]_{d_0-1}^{T}Y_{l}[x]_{d_0-1}), \\
			& X_{l} \in S_{+}^{d_0+1}, Y_{l} \in S_{+}^{d_0}, l = 1, \ldots, m.
		\earay \right.
	\end{equation}
	Its dual optimization is the moment relaxation
	\begin{equation}\label{5.5}
		\left\{ \baray{cl}
			\min & \langle f\; , \; y^{(1)}\rangle + \cdots + \langle f\; , \; y^{(m)}\rangle \\
			s.t. & (a_l+b_l)N_{d_0-1}[y^{(l)}] \succeq a_lb_lM_{d_0-1}[y^{(l)}]+G_{d_0}[y^{(l)}], \\
			& M_{d_0}[y^{(l)}] \succeq 0, \, y_0^{(1)} + \cdots + y_0^{(m)} = 1, \\
			& y^{(l)}=(y_{0}^{(l)}, y_{1}^{(l)}, \ldots, y_{2d_0}^{(l)}), \\
			& l = 1, \ldots, m.
		\earay \right.
	\end{equation}
	We still denote by $f_{sos}$ and $f_{mom}$ the optimal values of (\ref{5.4}) and (\ref{5.5}) respectively. The same conclusion in Theorem~\ref{thm 5.3} also holds here. The optimizers for \reff{5.1} can be obtained by the following algorithm.

	\begin{alg}\cite[Algorithm~3.3.6]{nie2023moment} \label{alg 5.6}\rm
		Assume $d=2d_0$ and $(y^{(1)}, \ldots, y^{(m)})$ is a minimizer for the moment relaxation (\ref{5.5}). For each $l$ with $y_0^{(l)} > 0$ and $r=\text{\normalfont rank}\,M_{d_0}[y^{(l)}]$, do the following:
		\begin{itemize}
			\item[Step~1] If $r\leq d_0$, solve the linear system
			$$\begin{bmatrix}
				y_0^{(l)} & y_1^{(l)} & \cdots & y_{r-1}^{(l)} \\
				y_1^{(l)} & y_2^{(l)} & \cdots & y_r^{(l)} \\
				\vdots & \vdots & \ddots & \vdots \\
				y_{2d_0-r}^{(l)} & y_{2d_0-r+1}^{(l)} & \cdots & y_{2d_0-1}^{(l)} \\
			\end{bmatrix}
			\begin{bmatrix}
				g_0^{(l)}  \\
				g_1^{(l)} \\
				\vdots \\
				g_{r-1}^{(l)} \\
			\end{bmatrix} =
			\begin{bmatrix}
				y_r^{(l)}  \\
				y_{r+1}^{(l)} \\
				\vdots \\
				y_{2d_0}^{(l)} \\
			\end{bmatrix}.$$
			\item[Step~2] If $r=d_0+1$, compute the smallest value of $y^{(l)}_{2d_0+1}$ satisfying
			\[
			b_lM_{d_0}[y^{(l)}] \succeq N_{d_0}[y^{(l)}] \succeq a_lM_{d_0}[y^{(l)}],
			\]
			then solve the linear system
			$$\begin{bmatrix}
				y_0^{(l)} & y_1^{(l)} & \cdots & y_{d_0}^{(l)} \\
				y_1^{(l)} & y_2^{(l)} & \cdots & y_{d_0+1}^{(l)} \\
				\vdots & \vdots & \ddots & \vdots \\
				y_{d_0}^{(l)} & y_{d_0+1}^{(l)} & \cdots & y_{2d_0}^{(l)} \\
			\end{bmatrix}
			\begin{bmatrix}
				g_0^{(l)}  \\
				g_1^{(l)} \\
				\vdots \\
				g_{d_0}^{(l)} \\
			\end{bmatrix} =
			\begin{bmatrix}
				y_{d_0+1}^{(l)}  \\
				y_{d_0+2}^{(l)} \\
				\vdots \\
				y_{2d_0+1}^{(l)} \\
			\end{bmatrix}.$$
			\item[Step~3] Compute $r$ distinct roots $t_1^{(l)}, \ldots, t_r^{(l)}$ of the polynomial
			\[
			g^{(l)}(x) \coloneqq g_0^{(l)}+g_1^{(l)}x+\cdots+g_{r-1}^{(l)}x^{r-1}-x^r.
			\]
			\item[Step~4] The roots $t_1^{(l)}, \ldots, t_r^{(l)}$ are minimizers of the optimization problem (\ref{5.1}).
		\end{itemize}
	\end{alg}
	The conclusion in Step~4 is justified by Theorem \ref{thm 5.8}.
	The following is an exposition for the above algorithm.
	\begin{eg}\rm
		Consider the constrained optimization problem
		\begin{equation}\notag
			\left\{ \baray{cl}
				\min & 4x^2 + 12x^3 + 13x^4 + 6x^5 + x^6\\
				s.t. & x \in [-4, -2] \cup [-1,2].
			\earay \right.
		\end{equation}
		The moment relaxation is
		\begin{equation}\notag
			\left\{ \baray{cl}
				\min & \langle f\; , \; y^{(1)}\rangle + \langle f\; , \; y^{(2)}\rangle \\
				s.t. & -6N_{2}[y^{(1)}] \succeq 8M_{2}[y^{(1)}]+G_{3}[y^{(1)}], \\
					  & N_{2}[y^{(2)}] \succeq -2M_{2}[y^{(2)}]+G_{3}[y^{(2)}], \\
					  & M_{3}[y^{(1)}] \succeq 0, M_{3}[y^{(2)}] \succeq 0, \\
					  & y_0^{(1)} + y_0^{(2)} =1.
			\earay \right.
		\end{equation}
		The minimizer $y^* = (y^{(*,1)}, y^{(*,2)})$ of the above is obtained as
		\begin{align*}
			y^{(*,1)} =&~0.0110 \cdot (1, -2, 4, -8, 16, -32, 64), \\
			y^{(*,2)} =&~0.9890 \cdot (1, 0, 0, 0, 0, 0, 0)~+ \\
			&~0.2190 \cdot (0, -1, 1, -1, 1, -1, 1).
		\end{align*}
		Applying Algorithm \ref{alg 5.6}, we get $g_0^{(1)} = -2, g_0^{(2)} = 0, g_1^{(2)} = -1$ and the polynomials
		\[
		g^{(1)}(x) = -2-x, \quad g^{(2)}(x) = -x - x^2.
		\]
		Therefore, the minimizers are the distinct roots $-2, -1, 0$ and the global minimum value $f_{min} = 0$.
	\end{eg}

	The performance of the moment relaxations \reff{5.3} and \reff{5.5} can be summarized as follows.
	
	\begin{theorem}\label{thm 5.8}
		Suppose $f$ is a univariate polynomial of degree $d=2d_0+1$ or $d=2d_0$.
		Then, all the optimal values $f_{min}$, $f_{sos}$, $f_{mom}$ are achieved for each corresponding optimization problem and they are all equal to each other.
		Suppose $y^* \coloneqq (y^{(*,1)}, \ldots, y^{(*,m)})$ is a minimizer of (\ref{5.3}) when $d=2d_0+1$ or of (\ref{5.5}) when $d=2d_0$.
		Then, the tms
		\[
		z^* \coloneqq  y^{(*,1)} + \cdots + y^{(*,m)}
		\]
		must admit a representing measure $\mu^*$  supported in $K$, and each point in the support of $\mu^*$ is a minimizer of (\ref{5.1}). If $f$ is not a constant polynomial, then $f$ has at most $2m +  \lceil (d-1)/2 \rceil $ minimizers and the representing measure $\mu^*$ for $z^*$ must be $r$-atomic with
        \[r \leq 2m + \lceil (d-1)/2 \rceil .\]
	\end{theorem}
	\begin{proof}
		Since each interval $[a_l,b_l]$ is compact, $K$ is also compact. So the minimum value $f_{min}$ is achievable, and it equals the largest $\gamma\in \mathbb{R}$ such that $f-\gamma$ is nonnegative on $[a_l,b_l]$ for every $l = 1, \ldots, m$, so $f_{min}=f_{sos}$ (see Theorem \ref{thm 5.3}). Each of the moment relaxations (\ref{5.3}) and (\ref{5.5}) has a strictly feasible point, e.g., the tms $\hat{y}^{(l)}=\int_{a_l}^{b_l} [x]_{2d_0+1} \,dx$
		is strictly feasible and
		\[
		b_lM_{d_0}[\hat{y}^{(l)}] \succ N_{d_0}[\hat{y}^{(l)}] \succ a_lM_{d_0}[\hat{y}^{(l)}].
		\]
		The tms $\Tilde{y}^{(l)}=\int_{a_l}^{b_l} [x]_{2d_0} \,dx$
		is strictly feasible and
		\[
		(a_l+b_l)N_{d_0-1}[\Tilde{y}^{(l)}] \succ a_lb_lM_{d_0-1}[\Tilde{y}^{(l)}]+G_{d_0}[\Tilde{y}^{(l)}], \quad M_{d_0}[\Tilde{y}^{(l)}] \succ 0.
		\]
		By the strong duality, $f_{sos}=f_{mom}$, and both (\ref{5.2}) and (\ref{5.4}) achieve their optimal values. By \cite[Theorem 3.3.4]{nie2023moment}, $y^{(*,l)}$ must admit a representing measure $\mu^{(l)}$ supported in $[a_l,b_l]$. Hence, $z^*$ must admit a representing measure $\mu^*=\mu^{(1)}+\cdots+\mu^{(m)}$ supported in $K$. The optimization problem (\ref{5.1}) is then equivalent to the linear convex conic optimization
		\begin{equation}\label{5.6}
			\left\{ \baray{cl}
				\min &  \int f\,d\mu  \\
				s.t. & \mu(K)=1, \mu \in \mathscr{B}(K),
			\earay \right.
		\end{equation}
		where $\mathscr{B}(K)$ denotes the convex cone of all Borel measures whose supports are contained in $K$. We claim that if a Borel measure $\mu^*$ is a minimizer of (\ref{5.6}), then each point in the support of $\mu^*$ is a minimizer of (\ref{5.1}). Suppose $E \subseteq K$ is the set of minimizers of (\ref{5.1}). For any $x^* \in E$, let $\dt_{x^*}$ denote the unit Dirac measure supported at $x^*$.
		Then, we have
		\[
		f_{min} = \int_{K} f_{min}\,d\mu^* \leq \int_{K} f(x)\,d\mu^* \leq \int_{K} f(x)\,d\dt_{x^*} = f(x^*) = f_{min}.
		\]
		Hence,
		\[
		0 = \int_{K} [f(x) - f_{min} ] \,d\mu^* = \int_{supp(\mu^*)} [f(x) - f_{min} ] \,d\mu^*.
		\]
		Thus, $f=f_{min}$ on $supp(\mu^*)$. This implies that $supp(\mu^*) \subseteq E$. So, every point in $supp(\mu^*)$ is a minimizer of \reff{5.1}.

		Note that $f$ has degree $d$. If $f$ is not a constant polynomial, it can have at most $d-1$ critical points. Moreover, the local maximizers and minimizers alternate. Thus, at most $\lceil (d-1)/2 \rceil$ of these critical points are local minimizers. On each interval $[a_l, b_l]$, two endpoints are possibly local minimizers. Since there are $m$ intervals in total, $f$ has at most $2m + \lceil (d-1)/2 \rceil$ local minimizers on $K$. In the above, we have proved that each point in $supp(\mu^*)$ is a minimizer of (\ref{5.1}). So the representing measure $\mu^*$ for $z^*$ must be $r$-atomic with $r \leq 2m + \lceil (d-1)/2 \rceil$.
	\end{proof}
	We refer to Algorithm \ref{alg 5.4} (when $d = 2d_0+1$) and Algorithm \ref{alg 5.6} (when $d=2d_0$) for how to determine the support of the representing measure $\mu^{(l)}$ for $y^{(*,l)}$. The points in the support are all minimizers of (\ref{5.1}). The upper bound for the number of minimizers is already sharp when $m=1$. For instance, consider the optimization
		\begin{equation}\notag
			\left\{ \baray{cl}
				\min & x(1-x)(x+1) \\
				s.t. & x \in [-1,1].
			\earay \right.
		\end{equation}
		There are $3$ global minimizers $-1, 0, 1$ and $ 2m + \lceil (d-1)/2 \rceil = 2+1 =3$.
		
We would like to remark that the representations for nonnegative univariate polynomials have broad applications.
In particular, it can be applied to the shape design of transfer functions for linear time invariant (LTI) single-input-single-output (SISO) systems \cite{nie2006shape}.
Since the transfer function is rational, the optimization problem can be formulated in terms of coefficients of polynomials.
We can then solve it by using representations of nonnegative univariate polynomials.
For instance, we look for a transfer function such that it is close to a piecewise constant shape.
That is, we want the transfer function to be close to given constant values $\xi_1, \ldots, \xi_m$ in a union of $m$ disjoint intervals $[a_l, b_l]$ with
\[
a_1 < b_1 < a_2 < b_2 < \cdots < a_m < b_m.
\]
As in \cite{nie2006shape}, the transfer function can be written as $p_1(x)/p_2(x)$. Specifically, we want to get $p_1, p_2$ such that
\[
\begin{gathered}
	p_1(x), p_2(x) \ge 0, \quad \forall x \ge 0, \\
	(1- \af)\xi_l \le \frac{p_1(x)}{p_2(x)} \le (1+\beta)\xi_l, \quad \forall x \in [a_l, b_l], \ l = 1, \ldots, m.
\end{gathered}
\]
The above is equivalent to the linear conic constraints
\begin{align*}
p_1(x), p_2(x) &\in \mathscr{P}_d([0, \infty)), \\
p_1 - (1-\af)\xi_lp_2 &\in \mathscr{P}_d([a_l, b_l]), \ l = 1, \ldots, m, \\
(1+\beta)\xi_lp_2 - p_1 &\in \mathscr{P}_d([a_l, b_l]), \ l = 1, \ldots, m.
\end{align*}
We refer to \cite{nie2006shape} for more details.

	\section{Numerical Experiments}\label{sc:eg}
	In this section, we present numerical experiments for how to solve polynomial optimization over the union of several basic closed semialgebraic sets. Algorithm~\ref{alg} is applied to solve it. All computations are implemented using MATLAB R2022a on a MacBook Pro equipped with Apple M1 Max processor and 16GB RAM. The unified moment relaxation (\ref{3.2}) is solved by the software \texttt{Gloptipoly} \cite{henrion2009gloptipoly}, which calls the SDP package \texttt{SeDuMi} \cite{sturm1999using}. For neatness, all computational results are displayed in four decimal digits.

	\begin{eg}\rm
		Consider the constrained optimization problem
		\begin{equation}\notag
		\left\{	\begin{array}{cl}
				\min\limits_{x \in \mathbb{R}^4} & (x_1^2 + x_2^2 + x_3^2 + x_4^2 + 1)^2 - 4(x_1^2x_2^2 + x_2^2x_3^2 + x_3^2x_4^2 + x_4^2 + x_1^2)\\
				s.t. & x \in K_1 \cup K_2 \cup K_3 \cup K_4,
			\end{array} \right.
		\end{equation}
		where
		\[K_1 = \{x \in \mathbb{R}^4 : x_1^2 + x_2^2 +x_3^2 \leq 0\},
		\]
		\[
		K_2 = \{x \in \mathbb{R}^4 : x_1^2 + x_2^2 +x_4^2 \leq 0\},
		\]
		\[
		K_3 = \{x \in \mathbb{R}^4 : x_1^2 + x_3^2 +x_4^2 \leq 0\},
		\]
		\[
		K_4 = \{x \in \mathbb{R}^4 : x_2^2 + x_3^2 +x_4^2 \leq 0\}.
		\]
		The objective function is a dehomogenization of the Horn's form \cite{Rez00}.
		For $k=2$, we get $f_{mom, 2} = 0$, and the flat truncation (\ref{3.3}) is met for all $l \in \mathcal{A}=\{1,4\}$. So,  $f_{mom,2}=f_{min}$.
		 The obtained four minimizers are
		\[
		(0, 0, 0, \pm1) \in K_1, \quad (\pm1, 0, 0, 0) \in K_4.
		\]
		For $k=2$, the unified moment relaxation \reff{3.2} took around $0.6$ second, while solving the individual moment relaxations \reff{1.3} for all $K_1, K_2, K_3, K_4$ took around $0.9$ second.
	\end{eg}

	\begin{eg}\rm
		Consider the constrained optimization problem
		\begin{equation}\notag
			\left\{ \baray{cl}
				\min\limits_{x \in \mathbb{R}^3} &  x_1^3 + x_2^3 + x_3^3 - x_1^2x_2 - x_1x_2^2 - x_1^2x_3 - x_1x_3^2  \\
				& \qquad \qquad \qquad \quad- x_2^2x_3 - x_2x_3^2 + 3x_1x_2x_3\\
				s.t. & x \in K_1 \cup K_2 \cup K_3,
			\earay \right.
		\end{equation}
		where
		\[
		K_1 = \{x \in \mathbb{R}^3 : x_1 \geq 0, x_1^2 + x_2^2 + x_3^2 =1\},
		\]
		\[
		K_2 = \{x \in \mathbb{R}^3 : x_2 \geq 0, x_1^2 + x_2^2 + x_3^2 =1\},
		\]
		\[
		K_3 = \{x \in \mathbb{R}^3 : x_3 \geq 0, x_1^2 + x_2^2 + x_3^2 =1\}.
		\]
		The objective function is obtained from Robinson's form \cite{Rez00} by changing $x_i^2$ to $x_i$ for each $i$.
		For $k=2$, we get $f_{mom, 2} = -1.3185$, and the flat truncation (\ref{3.3})
		is met for all $l \in \mathcal{A}=\{1,2,3\}$. So,  $f_{mom,2}=f_{min}$. The obtained three minimizers are
		\[
		(0.2783, 0.2783, -0.9193) \in K_1 \cap K_2,  \quad (0.2783, -0.9193, 0.2783) \in K_1 \cap K_3,
		\]
		\[
		(-0.9193, 0.2783, 0.2783) \in K_2 \cap K_3.
		\]
		For $k=2$, the unified moment relaxation \reff{3.2} took around $0.6$ second, while solving the individual moment relaxations \reff{1.3} for all $K_1, K_2, K_3$ took around $1.1$ seconds.
	\end{eg}

	\begin{eg}\rm
		Consider the constrained optimization problem
		\begin{equation}\notag
			\left\{ \baray{cl}
				\min\limits_{x \in \mathbb{R}^3} & x_1x_2x_3 + x_1^2x_2^2(x_1^2 + x_2^2) + x_3^6 - 3x_1^2x_2^2x_3^2\\
				s.t. & x \in K_1 \cup K_2 \cup K_3,
			\earay \right.
		\end{equation}
		where
		\[
		K_1 = \{x \in \mathbb{R}^3 : x_1^2+x_2^2-x_3^2 = 0, x_2x_3 \geq 0\},
		\]
		\[
		K_2 = \{x \in \mathbb{R}^3 : x_1^2+x_3^2-x_2^2 = 0, x_1x_3 \geq 0\},
		\]
		\[
		K_3 = \{x \in \mathbb{R}^3 : x_2^2+x_3^2-x_1^2 = 0, x_1x_2 \geq 0\}.
		\]
		The objective function is obtained from Motzkin's form \cite{Rez00} by adding the term $x_1x_2x_3$.
		For $k=3$, we get $f_{mom, 3} = -1.0757$, and the flat truncation (\ref{3.3})
		is met for all $l \in \mathcal{A} = \{2, 3\}$. So,  $f_{mom,3}=f_{min}$. The obtained four minimizers are
		\[
		(-1.0287, -1.6390, -1.2760) \in K_2, \quad (1.0287, -1.6390, 1.2760) \in K_2,
		\]
		\[
		(-1.6390, -1.0287, -1.2760) \in K_3, \quad (1.6390, 1.0287, -1.2760) \in K_3.
		\]
		For $k=3$, the unified moment relaxation \reff{3.2} took around $0.7$ second, while solving the individual moment relaxations \reff{1.3} for all $K_1, K_2, K_3$ took around $1.2$ seconds.
	\end{eg}

	\begin{eg}\rm
		Consider the constrained optimization problem
		\begin{equation}\notag
			\left\{ \baray{cl}
				\min\limits_{x \in \mathbb{R}^3} & x_1^2x_2^2 + x_1^2x_3^2 + x_2^2x_3^2 + 4x_1x_2x_3 \\
				s.t. & x \in K_1 \cup K_2 \cup K_3,
			\earay \right.
		\end{equation}
		where
		\[
		K_1 = \{x \in \mathbb{R}^3 : x_1 = x_2^2, x_3 = x_2^2 \},
		\]
		\[
		K_2 = \{x \in \mathbb{R}^3 : x_1^2+x_2^2+x_3^2 \leq 4, x_1x_2=-x_3, x_1x_3 \leq 0\},
		\]
		\[
		K_3 = \{x \in \mathbb{R}^3 : -1 \leq x_1 \leq 0, -1 \leq x_2 \leq 0, -1 \leq x_3 \leq 0 \}.
		\]
		The objective function is a dehomogenization of the Choi-Lam form \cite{Rez00}.
		For $k=2$, we get $f_{mom, 2} = -1$, and the flat truncation (\ref{3.3})
		is met for all $l \in \mathcal{A}=\{1,2,3\}$. So,  $f_{mom,2}=f_{min}$. The obtained four minimizers are
		\[
		(1, -1, 1) \in K_1, \quad (-1, 1, 1) \in K_2, \quad (1, 1, -1) \in K_2, \quad (-1, -1, -1) \in K_3.
		\]
		For $k=2$, the unified moment relaxation \reff{3.2} took around $0.6$ second, while solving the individual moment relaxations \reff{1.3} for all $K_1, K_2, K_3$ took around $1.1$ seconds.
	\end{eg}

	A class of problems like (\ref{1.1}) has absolute values in the constraints. For example, we consider that
	\[
	K = \Big\{ x : h(x) + \sum_{i=1}^{\ell} \lvert g_i(x) \rvert \geq 0 \Big\}.
	\]
	We can equivalently express $K$ as
	\be \label{K:abs}
	K = \bigcup_{s_1, \ldots, s_l = \pm 1}
	\Big\{ x : h(x) + \sum_{i=1}^{\ell}s_i \cdot g_i(x) \geq 0, s_i \cdot g_i(x) \geq 0 \Big\}.
	\ee

	\begin{eg}\label{eg abs val}\rm
		Consider the constrained optimization problem
		\begin{equation}\notag
			\left\{ \baray{cl}
				\min\limits_{x \in \mathbb{R}^2} & x_1^4 + x_2^4 - x_1^2x_2^2 - 2x_1^2 - 3x_2^2 \\
				s.t. & |x_1|^3 + |x_2|^3 \geq 4.
			\earay \right.
		\end{equation}
		The constraining set can be equivalently expressed as $\bigcup\limits_{l=1}^4 K_l$ with
		\[
		K_1 = \{x \in \mathbb{R}^2 : x_1 \geq 0, x_2 \geq 0, x_1^3 + x_2^3 \geq 4 \},
		\]
		\[
		K_2 = \{x \in \mathbb{R}^2 : x_1 \geq 0, -x_2 \geq 0, x_1^3 - x_2^3 \geq 4  \},
		\]
		\[
		K_3 = \{x \in \mathbb{R}^2 : -x_1 \geq 0, x_2 \geq 0, -x_1^3 + x_2^3 \geq 4  \},
		\]
		\[
		K_4 = \{x \in \mathbb{R}^2 : -x_1 \geq 0, -x_2 \geq 0, -x_1^3 - x_2^3 \geq 4  \}.
		\]
		A contour of the objective over the feasible set is in Figure~\ref{fig}.
		For $k=2$, we get $f_{mom, 2} = -6.3333$, and the flat truncation (\ref{3.3})
		is met for all $l \in \mathcal{A}=\{1,2,3,4\}$. So,  $f_{mom,2}=f_{min}$. The obtained four minimizers are
		\[
		(1.5275, 1.6330) \in K_1, \quad (1.5275, -1.6330) \in K_2,
		\]
		\[
		(-1.5275, 1.6330) \in K_3, \quad (-1.5275, -1.6330) \in K_4.
		\]
		For $k=2$, the unified moment relaxation \reff{3.2} took around $0.6$ second, while solving the individual moment relaxations \reff{1.3} for all $K_1, K_2, K_3, K_4$ took around $0.8$ second.
		
		\begin{figure}[htb]
			\centering
			\begin{tabular}{c}
				\includegraphics[width=0.5\textwidth]{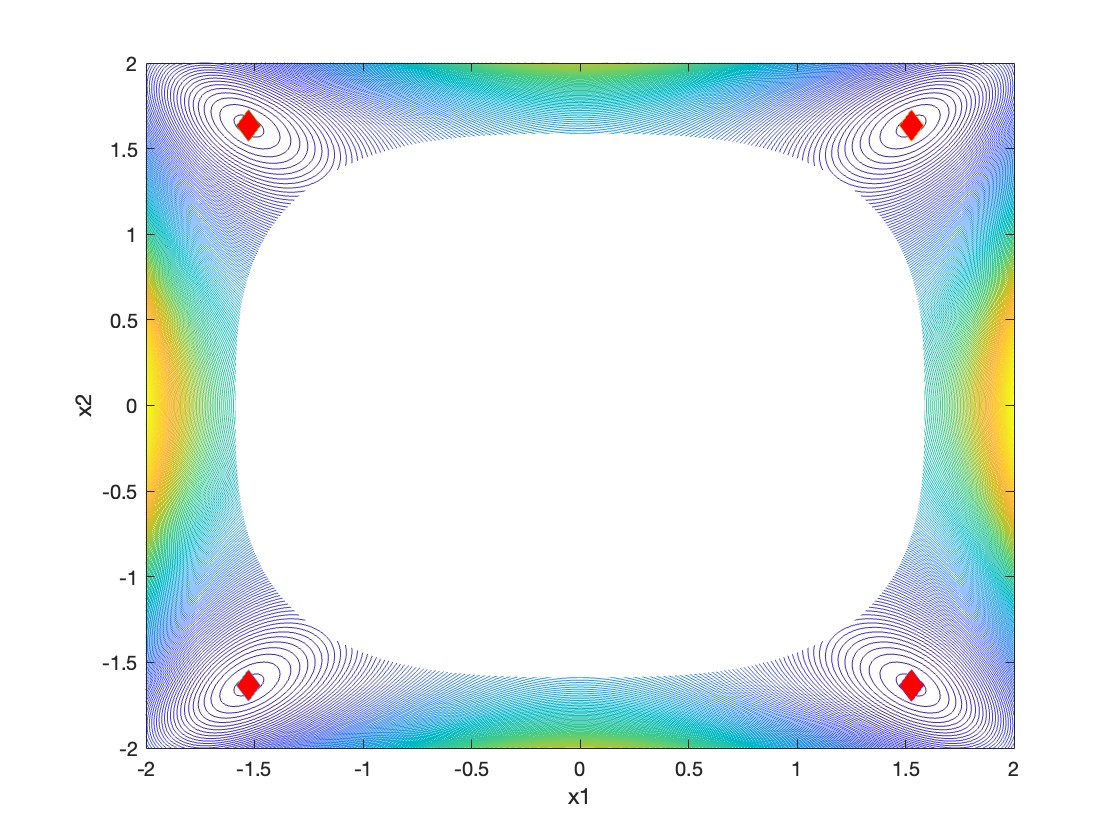}
			\end{tabular}
			\caption{The contour is for the objective function in Example~\ref{eg abs val}. The region outside the oval is the feasible set. The four diamonds are the minimizers.}\label{fig}
		\end{figure}
		
	\end{eg}

	Now we show how to compute the $(p,q)$-norm of a matrix $A \in \re^{m\times n}$ for positive integers $p,q$. Recall that
	\[
	\| A \|_{p,q} \coloneqq \max_{x \neq 0} \frac{\| Ax \|_{p}}{\|x\|_{q}} \ = \max_{\|x\|_q=1}\| Ax \|_{p}.
	\]
	When $p$ and $q$ are both even integers, this is a standard polynomial optimization problem. If one of them is odd, then we need to get rid of the absolute value constraints. When $p$ is even and $q$ is odd, we can equivalently express that
	\begin{equation}\label{norm_EO}
	\left\{	
		\begin{array}{rcl}
			(\| A \|_{p,q})^p = &  \max & (a_1^Tx)^p + \cdots + (a_m^Tx)^p \\
		 &	s.t. & |x_1|^q + \cdots + |x_n|^q =1.
		\end{array}
	\right.
	\end{equation}
	Here, the $a_i^T$ is the $i$th row of $A$. When $p$ is odd and $q$ is even, we have
	\begin{equation}\label{norm_OE}
		\left\{	
		\begin{array}{rcl}
			\| A \|_{p,q} = &\max & x_{n+1} \\
			&s.t. & (x_1)^q + \cdots + (x_n)^q =1, \\
			& &|a_1^Tx|^p + \cdots + |a_m^Tx|^p =  (x_{n+1})^p.
		\end{array}
		\right.
	\end{equation}
	Similarly, when $p$ and $q$ are both odd, we have
	\begin{equation}\label{norm_OO}
		\left\{	
		\begin{array}{rcl}
			\| A \|_{p,q} = &\max &  x_{n+1} \\
			&s.t. & |x_1|^q + \cdots + |x_n|^q =1, \\
			&& |a_1^Tx|^p + \cdots + |a_m^Tx|^p = (x_{n+1})^p.
		\end{array}
		\right.
	\end{equation}
	The constraining sets in the above optimization problems can be decomposed in the same way as in \reff{K:abs}.
	
	\begin{eg}\label{egnorm}\rm
		Consider the following matrix
		\[
		A = \begin{bmatrix*}[r]
		-8 &  -8 & -3 & 1 \\
		4  &  -7  &   7  &   6 \\
		6  &  -7  &  -7  &  -4 \\
		8   &  0  &  -9  &  -6
	    \end{bmatrix*}.
		\]
		Some typical values of the norm $\| A \|_{p,q}$ and the vector $x^*$
		for achieving it are listed in Table~\ref{tab:norm}.
			\begin{table}[htb]
			\begin{center}
			\caption{The $(p,q)$-norms for the matrix $A$ in Example~\ref{egnorm}.}
			\label{tab:norm}
			\begin{tabular}{||c| c | c ||}
				\hline
				$(p,q)$ &  $\|A\|_{p,q}$   & $x^*$ for $\|A\|_{p,q} = \|Ax^*\|_p, \|x^*\|_q=1$  \\ [0.5ex]
				\hline
				$(2,3)$ &  $21.6132$   &  $(0.6568,   -0.3937,   -0.7542,   -0.6097)$  \\
				\hline
				$(3,2)$ &  $15.5469$   &  $(0.5606, -0.2097, -0.6742, -0.4327)$  \\
				\hline
				$(3,3)$ &  $19.0928$   &  $(0.6782, -0.4598, -0.7329, -0.5820)$  \\
				\hline
				$(3,4)$ &  $21.2617$   &  $(0.7446, -0.5841, -0.7824, -0.6700)$  \\
				\hline
				$(4,3)$ &  $18.0128$   &  $(0.6825, -0.4605, -0.7305, -0.5794)$  \\
				\hline
				$(4,4)$ &  $ 20.0605$   &  $(0.7465, -0.5863, -0.7809, -0.6682)$  \\
				\hline
				$(4,5)$ &  $21.4196$   &  $(0.7895, -0.6633, -0.8166, -0.7261)$  \\
				\hline
				$(5,4)$ &  $19.3770$   &  $(0.7471, -0.5848, -0.7810, -0.6683)$  \\
				\hline
				$(5,5)$ &  $20.6894$   &  $(0.7896, -0.6635, -0.8165, -0.7260)$  \\
				\hline
			\end{tabular}
		\end{center}
		\end{table}
		The norms $\|A\|_{p,q}$ are all computed successfully by the unified moment relaxation \reff{3.2} for the relaxation order $k=2$ or $3$.

	\end{eg}

\section{Conclusions and Future Work}\label{sc:con}
This paper proposes a unified Moment-SOS hierarchy for solving the polynomial optimization problem \reff{1.1} whose feasible set $K$ is a union of several basic closed semialgebraic sets $K_l$.
Instead of minimizing the objective $f$ separately over each individual set $K_l$, we give a unified hierarchy of Moment-SOS relaxations to solve \reff{1.1}. This hierarchy produces a sequence of lower bounds for the optimal value $f_{min}$ of \reff{1.1}.
When the archimedeanness is met for each constraining subset $K_l$, we show the asymptotic convergence of this unified hierarchy.
Furthermore,  if the linear independence constraint qualification, the strict complementarity and the second order sufficient conditions hold at every global minimizer for each $K_l$, we prove the finite convergence of the hierarchy.
For the univariate case, special properties of the corresponding Moment-SOS relaxation are discussed.
To the best of the authors' knowledge, this is the first unified hierarchy of Moment-SOS relaxations for solving polynomial optimization over unions of sets.
Moreover, numerical experiments are provided to demonstrate the efficiency of this method.
In particular, as applications, we show how to compute the $(p,q)$-norm of a matrix for positive integers $p,q$.

There exists relevant work on approximation and optimization about measures with unions of several individual sets.
For instance, Korda et al. \cite{KLMS23} considers the generalized moment problem (GMP) that exploits the ideal sparsity,
where the feasible set is a basic closed semialgebraic set containing conditions like $x_ix_j = 0$.
Because of this, the moment relaxation for solving the GMP involves several measures, each supported in an individual set.
Lasserre et al. \cite{JBYE} proposes the multi-measure approach to approximate the moments of Lebesgue measures supported in unions of basic semialgebraic sets.
Magron et al. \cite{MFHsemid} discusses the union problem in the context of piecewise polynomial systems.
We would also like to compare the sizes of relaxations \reff{1.3} and \reff{1.5}.
To apply the individual relaxation \reff{1.3}, we need to solve it for $m$ times. For the unified relaxation \reff{1.5}, we only need to solve it for one time.
For a fixed relaxation order $k$ in \reff{1.3}, the length of the vector $y^{(l)}$ is {$n+2k \choose 2k$}.
For the same $k$ in \reff{1.5}, there are $m$ vectors of $y^{(l)}$, and each of them has length {$n+2k \choose 2k$}.
The comparison of the numbers of constraints is similar.
Observe that \reff{1.3} has $|\mc{E}^{(l)}|$ equality constraints, $|\mc{I}^{(l)}|+1$ linear matrix inequality constraints, and one scalar equality constraint.
Similarly, \reff{1.5} has $|\mc{E}^{(1)}|+\cdots+|\mc{E}^{(m)}|$ equality constraints,
$ |\mc{I}^{(1)}|+\cdots+|\mc{I}^{(m)}|+m$ linear matrix inequality constraints,
and one scalar equality constraint.
It is not clear which approach is more computationally efficient.
However, in our numerical examples, solving \reff{1.5} is relatively faster.

There is much interesting future work to do.
For instance, when the number of individual sets is large, the unified Moment-SOS relaxations have a large number of variables. How to solve the moment relaxation \reff{3.2} efficiently is important in applications.
For large scale problems, some sparsity patterns can be exploited. We refer to \cite{klep2022sparse, mai2023, wang2021ts, wang2022cs} for related work. It is interesting future work to explore the sparsity for unified Moment-SOS relaxations.
Moreover, how to solve polynomial optimization over a union of infinitely many sets is another interesting future work.

\subsection*{Acknowledgements}
Jiawang Nie is partially supported by the NSF grant DMS-2110780.
Linghao Zhang is partially supported by the UCSD School of Physical Sciences Undergraduate Summer Research Award.

\end{document}